\documentclass[12pt]{amsart}
\usepackage{amsfonts}
\usepackage{amsmath}
\usepackage{amssymb}
\usepackage[margin=1in]{geometry}
\usepackage{graphicx}
\usepackage{hyperref}
\usepackage{xcolor}
\usepackage{mathtools}

\newcommand{\R}{\mathbb{R}}
\newcommand{\C}{\mathbb{C}}

\renewcommand{\L}{L}

\newtheorem{theorem}{Theorem}[section]

\newtheorem{lemma}[theorem]{Lemma}

\DeclareMathOperator{\Real}{Re}
\DeclareMathOperator{\Imaginary}{Im}

\DeclareMathOperator\sign{sign}

\theoremstyle{definition}

\numberwithin{equation}{section}

\begin{document}
\title[Optimal decay of semi-uniformly stable semigroups]{Optimal decay of semi-uniformly stable operator semigroups with empty spectrum}

\author[M. Callewaert]{Morgan Callewaert}
\address{Department of Mathematics: Analysis, Logic and Discrete Mathematics\\ Ghent University\\ Krijgslaan 281\\ B 9000 Gent\\ Belgium}
\email{morgan.callewaert@UGent.be}

\author[L. Neyt]{Lenny Neyt}
\thanks{The research of L. Neyt was funded in whole by the Austrian Science Fund (FWF)
10.55776/ESP8128624. For open access purposes, the author has applied a CC BY public copyright
license to any author-accepted manuscript version arising from this submission.}
\address{University of Vienna\\Faculty of Mathematics\\Oskar-Morgenstern-Platz 1\\1090 Wien\\ Austria}
\email{lenny.neyt@univie.ac.at}

\author[J. Vindas]{Jasson Vindas}
\thanks{The work of J. Vindas was supported by Ghent University through the grant bof/baf/4y/2024/01/155}
\address{Department of Mathematics: Analysis, Logic and Discrete Mathematics\\ Ghent University\\ Krijgslaan 281\\ B 9000 Gent\\ Belgium}
\email{jasson.vindas@UGent.be}


\begin{abstract}
We show that it is impossible to quantify the decay rate of a semi-uniformly stable operator semigroup based on sole knowledge of the spectrum of its infinitesimal generator. 
More precisely, given an arbitrary positive function $r$ vanishing at $\infty$, we construct a Banach space $X$ and a bounded semigroup $ (T(t))_{t \geq 0}$ of operators on it whose infinitesimal generator $A$ has empty spectrum $\sigma(A)=\varnothing$, but for which, for some $x \in X$, 
$$
\limsup_{t\to\infty} \frac{\|T(t)A^{-1}x\|_{X}}{r(t)}=\infty.
$$
\end{abstract}

\maketitle

\section{Introduction}\label{Introduction}
Stability notions play a central role in the modern theory of $C_0$-semigroups. They lie at the heart of operator-theoretic methods for the asymptotic analysis of solutions to important classes of evolution equations.  We refer to \cite{A-B-H-N-VVLaplaceTransCauchyProb,C-T2007,vanNeervenbook} for excellent accounts on this traditional subject.

Significant progress has been made
in stability theory for operator semigroups in the past two decades. 
Many of its developments center around the concept of \emph{semi-uniform} stability and are intimately connected with complex Tauberian theorems for Laplace transforms of Ingham-Karamata type \cite{QuantifiedVersionsInghamTheorem,K-TaubTh}. 

Let \(\mathcal{T}=(T(t))_{t\geq 0}\) be a $C_{0}$-semigroup on a Banach space \(X\) and let $A$ be its infinitesimal generator. The semigroup is called semi-uniformly stable if $\mathcal{T}$ is (uniformly) bounded, $0\notin \sigma(A)$, and satisfies the decay condition
\begin{equation}
				\label{SemigroupLittleO}
				\lim_{t\to\infty}\| T(t) A^{-1} \|_{\L(X)} = 0.	
						\end{equation}
See the survey article \cite{C-S-T-SemiUniformStabOpSemigroupsEnergyDecayDampedWaves} for an overview of properties of such semigroups. A key feature of semi-uniformly stable semigroups is that they can be completely characterized in terms of the spectrum of their infinitesimal generators. The following precise description of semi-uniform stability goes back to the seminal work \cite{A-B88} (see also \cite{B-D-NonUniformStabBoundedSemiGroupsBanachSp,C-S-T-SemiUniformStabOpSemigroupsEnergyDecayDampedWaves}).

	\begin{theorem}
		\label{TheoremSemigroupQualitative}
		Let $\mathcal{T} $ be a bounded $C_{0}$-semigroup on a Banach space $X$ with generator $A$.
		Then, $\mathcal{T}$ is semi-uniformly stable 
		if and only if $\sigma(A) \cap i \R = \varnothing$.
	\end{theorem}

The natural question then becomes whether one can quantify the decay in \eqref{SemigroupLittleO}. The aim of this note is to show that it is impossible to deduce any quantified decay rate for a semi-uniformly stable operator semigroup from mere knowledge of the spectrum of its generator, so that the only admissible decay rate under hypotheses on the location of $\sigma(A)$ is $o(1)$ itself. Concretely, we shall prove the following result.

		\begin{theorem}
		\label{t:ExampleAbsenceOfRemainder}
     		Let \(r\) be a positive function tending to \(0\). There exists a bounded $C_{0}$-semigroup \(\mathcal{T}=(T(t))_{t\geq 0}\) on a certain Banach space \(X\), with infinitesimal generator $A$, such that \( \sigma(A) = \varnothing \) and, for some $x \in X$, 
			\begin{equation}
				\label{MainResultEstimation2}
         			\limsup_{t\to\infty} \frac{\|T(t)A^{-1} x\|_{X}}{r(t)}=\infty.     			\end{equation}
	\end{theorem}
	
It is worth comparing Theorem \ref{t:ExampleAbsenceOfRemainder} with a result of van Neerven \cite[Lemma 4.6]{vanNeerven95} (cf. \cite{meuller88}), which implies that if a semigroup is not uniformly stable (namely, $\|T(t)\|_{L(X)}\not\to 0$), then, given an arbitrary positive continuous function $r$ tending to $0$, one can always find an element $y \in X$ such that  
\begin{equation}
\label{eq:orbit non decay} \lim_{t\to\infty} \frac{\|T(t)y\|_{X}}{r(t)}=\infty.     	
\end{equation}
It should then be emphasized that, in general, \eqref{MainResultEstimation2} does not follow from \eqref{eq:orbit non decay} (unless one would additionally know that $y$ belongs to the domain of $A$, but such information is not delivered by van Neerven's result).

The proof of Theorem \ref{t:ExampleAbsenceOfRemainder} will be given in the next section. 
We will adjust a technique that originates from \cite[Section 4]{OptimalPolynomialDecay}.
The semigroup itself will always
 be\footnote{We point out that a reverse strategy, that is, fixing first a Banach space and then trying to find a semigroup so that the properties from Theorem \ref{t:ExampleAbsenceOfRemainder} are fulfilled, might fail depending on the chosen Banach space $X$ and the given function $r$. This is certainly the case for any Banach Grothendieck space $X$ with the Dunford-Pettis property if $|\log r(t)|=o(1)$, because on such spaces every semi-uniformly stable semigroup is uniformly (exponentially) stable, as follows from a result due to Lotz \cite[p.~56]{positive semigroups book} (cf. \cite{vanNeerven92}). Familiar examples of such Banach spaces are any $L^\infty$ space, the space of bounded harmonic functions on an open subset of $\mathbb{R}^{n}$, or the Hardy space $H^{\infty}(\Omega)$ on a finitely connected domain $\Omega\subset \C$; see \cite[p.~57]{positive semigroups book}.} the left translation semigroup acting on a certain Banach space $X$ of functions,
and with its generator being simply the derivative operator. 
The main difficulty lies in finding a suitable Banach space and an element in it on which the left translation semigroup achieves the desired properties.
Our constructions essentially rely on the scalar-valued counterpart of Theorem \ref{t:ExampleAbsenceOfRemainder}, which establishes the impossibility of obtaining remainders in the classical Ingham-Karamata theorem by just knowing the analytic continuation properties of Laplace transforms \cite{F-G-V-AbsRemIKandWI,D-V-NoteAbsenceRemWI}. We shall employ the version of such a result stated in \cite{C-N-V-NoteDensFuncEntLaplaceTrans}.

\begin{theorem}[{\cite[Theorem 1.2]{C-N-V-NoteDensFuncEntLaplaceTrans}}]
		\label{t:CounterExampleScalarCase}
		For any positive function $r$ tending to $0$, one can find a function $g \in C^1[0, \infty)$ such that $g'$ vanishes at $\infty$, its Laplace transform 
		
		\[
		\widehat{g}(\lambda)=\int_{0}^{\infty} g(t)e^{-\lambda t}dt, \qquad \Real{\lambda}>0,
		\]
admits entire extension to $\C$, and such that 
			\begin{equation}
				\label{eq:CounterExampleScalarCase} 
			\limsup_{t\to\infty} \frac{|g(t)|}{r(t)}=\infty.
			\end{equation}
	\end{theorem}
	
	Among other properties, the Banach space $X$ that we shall construct in Section \ref{sec:ProofMainThm} satisfies that both $g$ and $g'$ occurring in Theorem \ref{t:CounterExampleScalarCase}  are elements of  $X$, and, in turn, the relation  \eqref{eq:CounterExampleScalarCase} will yield \eqref{eq:orbit non decay}, because in fact we will choose $x = g'$ and be able to deduce that
\[
  \limsup_{t\to\infty} \frac{\|T(t) A^{-1} g'\|_{X}}{r(t)}=\limsup_{t\to\infty} \frac{\|T(t) g\|_{X}}{r(t)} =\infty.    
\]
	
	Finally, we should mention that, despite Theorem \ref{t:ExampleAbsenceOfRemainder}, it is actually possible to obtain quantified stability results for semigroups, which in many cases have been shown to be optimal or at least nearly optimal. However, instead of hypotheses on $\sigma(A)$, one needs to turn to resolvent bounds or related resolvent conditions. In fact, the underlying principle turns out to be: sharp resolvent bounds are (almost) equivalent to sharp decay rates. 
	There is a substantial body of literature on this topic; for instance, the reader can consult \cite{ B-D-NonUniformStabBoundedSemiGroupsBanachSp, OptimalPolynomialDecay, QuantifiedVersionsInghamTheorem, C-S-T-SemiUniformStabOpSemigroupsEnergyDecayDampedWaves, OptimalityQuantifiedVersionsInghamTheorem1, OptimalityQuantifiedVersionsInghamTheorem2, OptimalDecayHilbertSpaces, S-DecayC0Semigroups} for several results in this direction. We also refer to the extensive work \cite{debruyneGTIK}, where recently a general remainder Tauberian theory for the Ingham-Karamata theorem has been developed.

\section{The proof of Theorem \ref{t:ExampleAbsenceOfRemainder}}
\label{sec:ProofMainThm}

We move on to the proof of Theorem \ref{t:ExampleAbsenceOfRemainder}. The next technical lemma will play a role in our construction.
 We write $\overline{B}(z, \rho)$ for the closed ball centered around $z\in\mathbb{C}$ with radius $\rho > 0$.

\begin{lemma}\label{AuxiliaryLemma}
	Let $G$ be a continuous function on the strip $\{\lambda\in\C \mid |\Real {\lambda}|\leq1\}$.
	There exist a non-increasing function $\delta = \delta_G : [0, \infty) \to (0, 1)$ and a number $Q = Q_G > 0$ such that for every $v \in \R$ and $\lambda \in \overline{B}(iv, \delta(|v|))$ we have $|G(\lambda)| \leq |G(iv)| + 1$ and $\delta(|v|) / \delta(|\Imaginary{\lambda}|) \leq Q$.
\end{lemma}

\begin{proof}
	For any $a \geq 0$, we define the set
		\[ \Delta_{a} = \{ \rho \in(0,1] \mid \forall w \in [-a,a] , ~ \forall \lambda, \mu \in \overline{B}(i w, \rho) : ~ |G(\lambda)| \leq |G(\mu)| + 1 \} . \]
	Clearly $\Delta_{b} \subseteq \Delta_{a}$ whenever $a \leq b$. Also, each $\Delta_{a}$ is non-empty in view of the uniform continuity of $G$ on compact subsets.
	
	Now, we define our function $\delta = \delta_G$ as
		\[ \delta(a) = \frac{\sup \Delta_{a}}{2} , \qquad a \geq 0 . \] 
	So $\delta$ is a non-increasing function $[0, \infty) \to (0, 1/2]$. 

	Since for any $v \in \R$ we have $\delta(|v|) < \sup \Delta_{|v|}$, it immediately follows from the definition of $\Delta_{|v|}$ that $|G(\lambda)| \leq |G(iv)| + 1$ for each $\lambda \in \overline{B}(iv, \delta(|v|))$.
	It remains to verify the second claim.
	Set
		\[ Q = Q_G = \max \left\{ 2 , \frac{\delta(0)}{\delta(3/2)} \right\} . \]
	For any $v \in \R$ with $|v| \leq 1$ and $\lambda \in \overline{B}(iv, \delta(|v|)) \subseteq \overline{B}(iv, 1/2)$, we have $|\Imaginary{\lambda}| \leq 3/2$, so in particular $\delta(|v|) / \delta(|\Imaginary{\lambda}|) \leq Q$.
	Hence, we are left with the case $|v| > 1$.
	 If $|\Imaginary{\lambda}| \leq |v|$ and $\lambda \in \overline{B}(iv, \delta(|v|))$, then trivially $\delta(|v|) / \delta(|\Imaginary{\lambda}|) \leq 1 \leq Q$.
	So, we may assume that $|\Imaginary{\lambda}| > |v|$.
	In this case, we claim that
		\begin{equation} 
			\label{eq:ClaimDelta}
			0 < 2 \delta(|v|) - |\Imaginary{\lambda} - v| =: \rho_\lambda \in \Delta_{|\Imaginary{\lambda}|} . 
		\end{equation}
	If this were true, then,
		\begin{align*} 
			2 \delta(|\Imaginary{\lambda}|) 
			&= \sup \Delta_{|\Imaginary{\lambda}|}
			\geq \rho_\lambda
			= 2 \delta(|v|) - |\Imaginary{\lambda} - v|
			\geq \delta(|v|) ,
		\end{align*}
	showing that $\delta(|v|) / \delta(|\Imaginary{\lambda}|) \leq Q$.
	
We now verify that \eqref{eq:ClaimDelta}
holds true for any $|v| > 1$ and $\lambda \in \overline{B}(iv, \delta(|v|))$ with $|\Imaginary{\lambda}| > |v|$.
Fix such $v$ and $\lambda$.
 Then, the numbers $\Imaginary{\lambda}$ and $v$ have the same sign, so that $|\Imaginary{\lambda} - v| = |\Imaginary{\lambda}| - |v|$.
	Take any $|w| \leq |\Imaginary{\lambda}|$ and $\xi, \zeta \in \overline{B}(iw, \rho_\lambda)$.
	If $|w| \leq |v|$ then $|G(\xi)| \leq |G(\zeta)| + 1$ since $\rho_{\lambda} < 2 \delta(|v|) = \sup \Delta_{|v|}$.
	Suppose now that $|v| \leq |w| \leq |\Imaginary{\lambda}|$.
	Then,
		\[
			| \xi - \sign(w) i |v|| \leq |\xi - iw| + |i w - \sign(w) i |v|| \leq \rho_\lambda + |w| - |v| \leq 2 \delta(|v|) ,
		\]
	so that $\xi$, and analogously $\zeta$, is contained in $\overline{B}(\sign(w) i |v|, 2 \delta(|v|)) =\overline{B}(\sign(w) i |v|, \sup \Delta_{|v|})$.
	Consequently, using that $\sup \Delta_{|v|}\in \Delta_{|v|}$, which follows from the continuity of $G$, also here $|G(\xi)| \leq |G(\zeta)| + 1$, and we may conclude that $\rho_\lambda \in \Delta_{|\Imaginary{\lambda}|}$, thereby finishing the proof.
\end{proof}

We are now sufficiently prepared for our proof of Theorem \ref{t:ExampleAbsenceOfRemainder}.

\begin{proof}[Proof of Theorem \ref{t:ExampleAbsenceOfRemainder}] 
Let $g$ be as in Theorem \ref{t:CounterExampleScalarCase} for $r$ and let $\delta$ be the function and $Q \geq 1$ the constant as in Lemma \ref{AuxiliaryLemma} for $G = \widehat{g}$.
Consider the vector space $X$ of all bounded uniformly continuous functions $f : [0, \infty) \to \C$ whose Laplace transform $\widehat{f}$ analytically extends to the whole of $\C$ and satisfies
	\begin{equation}
		\label{eq:SpaceXLimitCond} 
		\frac{| \widehat{f}(\lambda) |}{M(\lambda)} \to 0 , \qquad \text{as } |\lambda| \to \infty , 
	\end{equation}
where
	\[ M(\lambda) = \frac{(1 + |\lambda|)^{2} \sup_{|\Real{\mu}| \leq |\Real{\lambda}|, \Imaginary{\mu} = \Imaginary{\lambda}} (1 + |\widehat{g}(\mu)|)}{\delta(|\Imaginary{\lambda}|)} , \qquad \lambda \in \C . \]
Notice $M(\lambda)\geq 1$. We endow $X$ with the norm
	\[ \| f \|_{X} = \|f\|_{L^{\infty}} + \|\widehat{f}\|_{L^{\infty}_{M}} := \|f\|_{L^{\infty}(\R_+)} + \sup_{\lambda \in \C} \frac{\vert\widehat{f}(\lambda)\vert}{M(\lambda)} , \]
turning it into a Banach space. Note that, for any bounded measurable function $f : [0, \infty) \to \C$ and $\lambda \in \C$ with $\Real{\lambda} > 0$, we have
	\[ |\widehat{f}(\lambda)| \leq \frac{\|f\|_{L^{\infty}}}{\Real{\lambda}} . \]
Consequently,
	\[ \|\widehat{f}\|_{L^{\infty}_{M}} \leq \|f\|_{L^{\infty}} + \sup_{\Real{\lambda} < 1} \frac{\vert\widehat{f}(\lambda)\vert}{M(\lambda)} , \]
and we just have to verify \eqref{eq:SpaceXLimitCond} for $\Real{\lambda} < 1$, since $M(\lambda) \to \infty$ as $|\lambda| \to \infty$.

Let us consider the left translation semigroup $\mathcal{T} = (T(t))_{t \geq 0}$ on $X$, i.e., $T(t) f(s) = f(s+ t)$.
We start by verifying that it defines a $C_{0}$-semigroup on $X$.
Indeed, for a fixed $t \geq 0$ clearly $T(t) f$ is a bounded uniformly continuous function with $\|T(t) f\|_{L^{\infty}} \leq \|f\|_{L^{\infty}}$. Moreover, for any $\lambda \in \C$ with $\Real{\lambda} > 0$, we have
	 \begin{equation}
	 	\label{LaplaceOfSemigroup}
    		\widehat{T(t)f}(\lambda) = e^{\lambda t} \widehat{f}(\lambda) - e^{\lambda t} \int^{t}_0 f(s)e^{-\lambda s}\, ds.
    	\end{equation}
The right-hand side of \eqref{LaplaceOfSemigroup} extends analytically to the whole of $\C$ and, for $\Real{\lambda} < 1$,
	\[ \frac{|\widehat{T(t)f}(\lambda)|}{M(\lambda)} \leq e^{t} \frac{|\widehat{f}(\lambda)|}{M(\lambda)} + t e^{t} \frac{\|f\|_{L^{\infty}}}{M(\lambda)} \leq (1 + t) e^t \|f\|_{X} . \]
From here, we may deduce that $T(t) \in L(X)$. 
To show that $\mathcal{T}$ is strongly continuous, it suffices to verify that, for a fixed $f\in X$, we have $\|T(t) f - f\|_{X} \to 0$ as $t \rightarrow 0^{+}$.
Since $f$ is uniformly continuous, it is clear that $\|T(t) f - f\|_{L^{\infty}} \to 0$ as $t \rightarrow 0^{+}$.
On the other hand, by \eqref{LaplaceOfSemigroup}, we find, for arbitrary $\rho>0$,
\begin{align*}
     \limsup_{t\to0^{+}}  & \sup_{\Real{\lambda} < 1}  \frac{|\widehat{T(t)f}(\lambda) - \widehat{f}(\lambda)|}{M(\lambda)}
        &
         \\
        &\leq      \limsup_{t\to0^{+}}  \sup_{\substack{\Real{\lambda} < 1\\ |\lambda| \geq \rho}}\frac{\vert e^{\lambda t}-1\vert \vert\widehat{f}(\lambda)\vert }{M(\lambda)} +   \lim_{t\to0^{+}}\sup_{\substack{\Real{\lambda} < 1\\ \vert\lambda| \leq \rho}}\frac{\vert e^{\lambda t}-1\vert \vert\widehat{f}(\lambda)\vert }{M(\lambda)} +\lim_{t\to 0^{+}} te^{t}\Vert f\Vert_{L^{\infty}}  \\
        &= \limsup_{t\to0^{+}} \sup_{\substack{\Real{\lambda} < 1 \\ |\lambda| \geq \rho}}\frac{\vert e^{\lambda t}-1\vert \vert\widehat{f}(\lambda)\vert }{M(\lambda)} \leq 2 \sup_{\substack{\Real{\lambda} < 1 \\ |\lambda| \geq \rho}}\frac{ \vert\widehat{f}(\lambda)\vert }{M(\lambda)}.
    \end{align*}
    Taking $\rho\to\infty$, we obtain $\|T(t) f - f\|_{X} \to 0$ as $t \to 0^{+}$, and thus $\mathcal{T}$ is a $C_{0}$-semigroup on $X$.
Its generator $A$ is given by
	\[ A f = f' , \]
with domain 
    	\[ D(A) = \{ f \in X \cap C^{1}[0,\infty) \mid f' \in X \} . \]
This follows exactly as in \cite[Chapter II, Proposition 1, p.~51]{course-operatorsemigroups}.

Next, we verify that $\mathcal{T}$ is a bounded $C_{0}$-semigroup on $X$.
Obviously, as noted before, $\|T(t) f\|_{L^{\infty}} \leq \|f\|_{L^{\infty}}$ for any $f \in X$. Therefore, we only need to show that there exists some constant $C > 0$ such that
	 \[ \|\widehat{T(t)f}\|_{L^{\infty}_{M}} \leq C \|f\|_{X}, \qquad \forall f \in X , ~ t \geq 0 . \]
For fixed $f \in X$ and $t \geq 0$, we consider the entire function $F = \widehat{T(t) f}$, which satisfies
	\[
		| F(\lambda) | 
		\leq
		|\widehat{f}(\lambda)| + \frac{\|f\|_{L^{\infty}}}{|\Real{\lambda}|} , \qquad \Real{\lambda} \neq 0 .
	\]
Hence, for every $v \in \R$ and $\lambda \in \overline{B}(iv, \frac{1}{M(iv)})$ with $\Real{\lambda} \neq 0$,
	\[  |F(\lambda)| \leq \frac{\|f\|_X}{|\Real{\lambda}|} \left(1 + \frac{M(\lambda)}{M(iv)}\right) . \]
Now, since $\overline{B}(iv, \frac{1}{M(i v)}) \subseteq \overline{B}(iv, \delta(|v|)) \subseteq \overline{B}(iv, 1)$, we find
	\[ \frac{M(\lambda)}{M(iv)} = \frac{(1 + |\lambda|)^2}{(1 + |v|)^2} \cdot \frac{\sup_{|\Real{\mu}| \leq |\Real{\lambda}|, \Imaginary{\mu} = \Imaginary{\lambda}} 1 + |\widehat{g}(\mu)|}{1 + |\widehat{g}(iv)|} \cdot \frac{\delta(|v|)}{\delta(|\Imaginary{\lambda}|)} \leq 8 Q. \]
Thus

	\[ |F(\lambda)| \leq 9 Q \frac{\|f\|_X}{|\Real{\lambda}|} \]	
for every $v \in \R$ and $\lambda \in \overline{B}(iv, \frac{1}{M(iv)})$ with $\Real{\lambda}\neq 0$.
Applying Levinson's lemma \cite[Lemma 4.6.6, p.~297]{A-B-H-N-VVLaplaceTransCauchyProb}, we get that
	\[ |F(\lambda)| \leq 24 Q \|f\|_{X} M(i v) , \]
for all $v \in \R$ and $\lambda \in \overline{B}(iv,\frac{1}{2 M(iv)})$.
Using the above estimates, for any $\lambda \in \C$,
\begin{align*}
        \| \widehat{T(t) f} \|_{L^{\infty}_{M}}
        &\leq \sup_{\substack{\lambda \in \C \\ |\Real{\lambda}| \leq \frac{1}{2 M(i \Imaginary{\lambda})}}} \frac{|F(\lambda)|}{M(\lambda)} + \sup_{\substack{\lambda \in \C \\ |\Real{\lambda}| \geq \frac{1}{2 M(i \Imaginary{\lambda})}}} \frac{|F(\lambda)|}{M(\lambda)}
        \\ 
        &\leq 24 Q \sup_{\substack{\lambda \in \C \\ |\Real{\lambda}| \leq \frac{1}{2 M(i \Imaginary{\lambda})}}} \frac{M(i \Imaginary{\lambda})}{M(\lambda)} \|f\|_{X} + \sup_{\substack{\lambda \in \C \\ |\Real{\lambda}| \geq \frac{1}{2 M(i \Imaginary{\lambda})}}} \frac{|\widehat{f}(\lambda)| + \|f\|_{L^{\infty}} |\Real{\lambda}|^{-1}}{M(\lambda)}
   	\\ 
    	&\leq \|f\|_{X} \left[1 +  (2 + 24Q) \sup_{\lambda \in \C} \frac{M(i \Imaginary{\lambda})}{M(\lambda)} \right]
    	\leq (3 + 24 Q) \|f\|_{X} . 
\end{align*}
We can thus conclude that $\mathcal{T}$ is a bounded $C_{0}$-semigroup on $X$.

Let us demonstrate that $\sigma(A) = \varnothing$.
As $\mathcal{T}$ is bounded, we already have $\sigma(A) \subseteq \{ \lambda \in \C \mid \Real{\lambda} \leq 0 \}$.
Fix now any $\lambda \in \C$ with $\Real{\lambda} \leq 0$.
Because the generator \((A,D(A))\) is closed, it suffices to show that $\lambda - A : D(A) \to X$ is bijective.
Injectivity follows immediately: If $f \in D(A)$ is such that $(\lambda - A) f = 0$, then, $f(s)= c e^{\lambda s}$ for some $c \in \C$, which is only possible if $c = 0$ (since $\widehat{f}$ must be entire).
We now consider surjectivity.
Let $f \in X$ be fixed.
We claim that 
	\[ h(t) = e^{\lambda t} \widehat{f}(\lambda) - e^{\lambda t} \int^t_0 f(s) e^{-\lambda s} \, ds \] 
is the pre-image of $f$ under $\lambda-A$.
To show this, we need to prove that $h \in D(A)$ and $(\lambda - A) h = f$.
Now, $h'(t) = \lambda h(t) - f(t)$, so that $h \in C^{1}[0, \infty)$ and $(\lambda - A) h =f$. Moreover, if $h \in X$, then also $h' \in X$, so that in this case $h \in D(A)$.
It remains to verify that $h \in X$.
First, $h$ is bounded and uniformly continuous. Indeed, if $\Real{\lambda} = 0$
        \[ |h(t)| \leq |\widehat{f}(\lambda)| + \left| \int^t_0 e^{-\lambda s}f(s)\, ds \right| , \]
and the latter integral is bounded because of \cite[Theorem 1.1]{D-V-OptTaubConstIngham},
while if $\Real{\lambda}<0$
        \[ |h(t)| \leq |\widehat{f}(\lambda)| +\frac{\|f\|_{L^{\infty}}}{|\Real{\lambda}|} . \] 
Consequently, $h$ is bounded. As $h' = \lambda h - f$ is then also bounded, it follows that $h$ is uniformly continuous.
Using Fubini's theorem, one calculates that the Laplace transform of $h$ for $\Real{\mu} > 0$ is
    \[ \widehat{h}(\mu) = -\frac{\widehat{f}(\lambda) - \widehat{f}(\mu)}{\lambda-\mu} . \]
This function analytically extends to the whole of $\C$ because $f \in X$ and $\mu = \lambda$ is just a removable singularity.
Finally,
	\[ \frac{|\widehat{h}(\mu)|}{M(\mu)} =  \left| \frac{\widehat{f}(\lambda) - \widehat{f}(\mu)}{(\lambda-\mu ) M(\mu)} \right| \to 0, \qquad \text{as } |\mu| \to \infty . \]
Hence, $h \in X$, and we have shown that $\sigma(A) = \varnothing$.

Finally, we prove that \eqref{MainResultEstimation2} holds for $\mathcal{T}$ with $x = g^\prime$.
Note that $g$ and $g'$ both belong to $X$, where the boundedness of $g$ follows from the classical Ingham-Karamata theorem. 
In view of \eqref{eq:CounterExampleScalarCase}, we then have
\[ \infty= \limsup_{t\to\infty} \frac{|g(t)|}{r(t)} \leq \limsup_{t \to \infty} \frac{\|T(t) g\|_X}{r(t)} =  \limsup_{t\to\infty} \frac{\|T(t)A^{-1}g'\|_{X}}{r(t)} , \] 
completing the proof of Theorem \ref{t:ExampleAbsenceOfRemainder}.
\end{proof}

\end{document}